\documentclass[10 pt]{article}  
\usepackage[utf8]{inputenc}
\usepackage{amsmath}
\usepackage{amsfonts}
\usepackage{amssymb}
\usepackage{graphicx}
\usepackage{mathrsfs}
\usepackage{upref,amsthm,amsxtra,exscale}
\usepackage{stmaryrd}
\usepackage{cite}
\usepackage[colorlinks=true,urlcolor=blue,
citecolor=red,linkcolor=blue,linktocpage,pdfpagelabels,
bookmarksnumbered,bookmarksopen]{hyperref}
\usepackage{doi}
\usepackage{todonotes}

\usepackage[cm]{fullpage}
\usepackage{subcaption}
\usepackage{caption}
\usepackage{cleveref}
\usepackage{enumitem}
\usepackage{authblk}

\newcommand{\pv}{\operatorname{p.\!v.}}
\providecommand{\R}{\mathbb{R}}

\providecommand{\N}{\mathbb{N}}

\def\eps{\varepsilon}

\newtheorem{theorem}{Theorem}[section]

\newtheorem{lemma}[theorem]{Lemma}

\numberwithin{equation}{section}

\theoremstyle{definition}
\newtheorem{remark}[theorem]{Remark}

\title{A note on a Pohozaev identity for the fractional Green function}
\author[1]{Abdelrazek Dieb}
\author[2]{Isabella Ianni}
\affil[1]{\small{University Ibn Khaldoun of Tiaret, Algeria}}
\affil[2]{\small{{\sl Sapienza} Universit\`a di Roma, Italy}}
\date{}

\begin{document}

\maketitle

\begin{abstract}
We get a Pohozaev-type identity for the fractional Green function, which extends to the fractional setting a classical  result by Brezis and Peletier \cite{BrezisPeletier}. Our result complements with the recent ones in \cite{DjitteSueur} concerning a representation formula for the gradient of the fractional Robin function.
\vskip 2mm

\vskip 2mm

\textbf{Acknowledgements:} This work is supported by the 
ICTP-INdAM Collaborative Grants and Research in Pairs Program

\end{abstract}

\section{Introduction}\label{s:1}

Let us consider the  fractional Laplace operator of order $s$, with $s\in (0,1)$, 
$$
(-\Delta)^su (z) := c_{N,s}\,  \pv \int_{\R^N}\frac{u(z)-u(y)}{|z-y|^{N+2s}}dy,
$$
with  normalization constant given by
$$c_{N,s} :=\pi^{-\frac{N}{2}}s4^s\frac{\Gamma(\frac{N+2s}{2})}{\Gamma(1-s)},$$
where $\Gamma$ is the gamma function, $\pv$ refers to the Cauchy principal value, and $N\in\mathbb N^*$. For $x\in\Omega$, we denote by $G_s(x,\cdot)$ the Green function associated to the operator $(-\Delta)^s$ in the domain $\Omega$, that is, the solution to \begin{equation} \label{DIR}
\left\{ \begin{array}{rcll}
    (-\Delta)^sG_s(x,\cdot)&=&\delta_x& \quad\text{in}\quad \mathcal D'(\Omega)\\
    G_s(x,\cdot)&=&0&\quad\text{in}\quad \R^N\setminus\Omega,
    \end{array} 
    \right.
\end{equation}
where $\delta_x$ denotes the Dirac delta distribution at $x$. The Green function $G_s(x,\cdot)$ can be split into
\begin{equation}\label{Eq-spliting-of-Green}
    G_s(x,\cdot)= F_s(x,\cdot)-H_s(x,\cdot),
\end{equation}
where 
\begin{equation}\label{funda}
   F_s(x,\cdot) :=\frac{b_{N,s}}{|x-\cdot|^{N-2s}}, 
\end{equation}
 is the  fundamental solution of $(-\Delta)^s$, with $b_{N,s}:=-c_{N,-s} (>0)$,
 hence 
$H_s(x,\cdot)$,  the regular part of the fractional Green function $G_s(x,\cdot)$,  solves the equation
\begin{equation}\label{regular-part-of-Green}
\left\{ \begin{array}{rcll} (-\Delta)^s H_s(x,\cdot)&=& 0  &\textrm{in }\Omega, \\ H_s(x,\cdot)&=&F_s(x,\cdot)&
\textrm{in }\R^N\setminus\Omega. \end{array}\right. 
\end{equation}
We also recall that the fractional Robin function associated with the domain $\Omega$, denoted as $\mathcal R_s:\Omega\to \R$ is given for $x$ in $\Omega$ by 
\begin{equation}\label{def-s-robin}
\mathcal R_s(x):= H_s(x,x).
\end{equation}
In \cite[Theorem 4.3]{BrezisPeletier} it is proved that the Robin function $\mathcal R_1$ of $\Omega$ associated with the classical Laplacian (namely in the case $s=1$) satisfies the following Pohozaev-type identity
     \begin{align}\label{repres-class-Robin}
         \mathcal R_1(x)=\frac{1}{N-2}\int_{\partial\Omega}(\partial_\nu G_1(x,\cdot))^2\langle \cdot-x,\nu \rangle \, d\sigma,\qquad\mbox{for any }x\in\Omega,
     \end{align}
where $G_1$ is the Green function in the case of the Laplacian, and $N>2$. In this note we derive the analogous representation formula for the  fractional Robin function $\mathcal R_s$, when $s\in (0,1)$. 
\begin{theorem}
\label{MainTheorem}Let $N\in\mathbb N$, $s\in (0,1)$, $N>2s$ and let $\Omega$ be a bounded open set of $\mathbb R^N$ of class $C^{1,1}$. 
Then, for every $x\in\Omega$ we have
\begin{equation}
\label{PohozaevGreen}
\mathcal R_s(x)=\frac{\Gamma^2(1+s)}{N-2s}\int_{\partial\Omega}\left( \frac{G_s(x, \cdot)}{\delta^s}\right)^2 \langle \cdot-x,\nu \rangle \, d\sigma,
\end{equation}
where $\nu$ is the outward unit normal to $\partial\Omega$ and $\delta:=dist (\ \cdot\ , \mathbb R^N\setminus\Omega)$.
\end{theorem}
The identity \eqref{PohozaevGreen}  generalizes to the fractional setting the classical Pohozaev-type identity for the Green function \eqref{repres-class-Robin} obtained  by Brezis and Peletier in the case $s=1$.  Our result complements the one in \cite{DjitteSueur}, where among other things,  a representation formula for the \emph{gradient} of the fractional Robin function $\mathcal R_s$ has been derived (see \cite[Theorem 1.4]{DjitteSueur}), extending analogous classical results to the fractional case.\\\\
Properties of Green and Robin functions are extremely useful in PDEs. The classical identity \eqref{repres-class-Robin} plays an important role, for instance, in the study of  concentration phenomena for elliptic boundary value problems involving the critical Sobolev exponent (see \cite{BrezisPeletier, Han, Grossi}). Identity \eqref{PohozaevGreen} could hence be applied to get sharp asymptotic characterizations of solutions in the fractional setting, both improving existing results in the literature (like the ones in \cite{Baktha_Muk_Santra}, for instance), and to get new sharp results (see the forthcoming paper \cite{DeMarchisIanniSaldana}). \\\\
Theorem \ref{MainTheorem} is obtained as a corollary of the  identity \eqref{bilinearPohozaevGreen} contained in the next result:
\begin{theorem}
\label{MainTheoremBilinear}
Let $N\in\mathbb N$, $s\in (0,1)$, $N>2s$ and let $\Omega$ be a bounded open set of $\mathbb R^N$ of class $C^{1,1}$. Let $x,y\in\Omega$ with $x\neq y$. Then 
\begin{equation}\label{bilinearPohozaevGreen}\Gamma^2(1+s)\int_{\partial\Omega}\frac{ G_s(x,\cdot)}{\delta^s}\frac{G_s(y,\cdot)}{\delta^s} \langle \cdot-x,\nu\rangle \, d\sigma=(N-2s) H_s(x,y) +   \langle \nabla_y  H_s(x,y),y-x\rangle,
\end{equation}
where $\nu$ is the outward unit normal to $\partial\Omega$ and $\delta:=dist (\ \cdot\ , \mathbb R^N\setminus\Omega)$. \\
Furthermore, if $s>\frac{1}{2}$, then  for every $\xi\in\mathbb R^N$ there holds
\begin{equation}\label{bilinearPohozaevGreenGENERAL}\Gamma^2(1+s)\int_{\partial\Omega}\frac{ G_s(x,\cdot)}{\delta^s}\frac{G_s(y,\cdot)}{\delta^s} \langle \cdot-\xi,\nu\rangle \, d\sigma=(N-2s) H_s(x,y)
 +  \langle \nabla_x  H_s(y,x),x-\xi\rangle +  \langle \nabla_y  H_s(x,y),y-\xi\rangle.
\end{equation}
\end{theorem}
Let us observe that for $s>\frac{1}{2}$ one could derive \eqref{bilinearPohozaevGreen} as a special case of \eqref{bilinearPohozaevGreenGENERAL} taking $\xi=x$. However,  differently from \eqref{bilinearPohozaevGreenGENERAL},  \eqref{bilinearPohozaevGreen} is obtained in the full range $s\in (0,1)$.\\ 

The proof of Theorem \ref{MainTheoremBilinear}  starts from the following generalized fractional integration-by-parts formula, which has been derived in \cite{DjitteFallWeth}, for suitably smooth functions $u,v$: 
\begin{align} \label{poho_DFW}
\int_\Omega\langle \mathcal X,\nabla u\rangle(-\Delta)^{s}v\,dx+\int_\Omega \langle \mathcal X,\nabla v\rangle(-\Delta)^{s}u\,dx=-\mathcal E_{\mathcal X}(u,v)-\Gamma(1+s)^{2}\int_{\partial\Omega}\frac{u}{\delta^{s}}\frac{v}{\delta^{s}}\langle \mathcal X,\nu\rangle\,d\sigma,
\end{align}
where  $\mathcal X$ is a given Lipschitz vector field on $\mathbb R^N$, and
\[
\mathcal E_{\mathcal X}(v,u):=\iint_{\R^N\times\R^N}\big(v(z)-v(t)\big)\big(u(z)-u(t)\big)K_\mathcal X(z,t)\,dzdt,
\]
with kernel 
\[K_{\mathcal X}(z,t):= \frac{c_{N,s}}{2}\Bigl[\bigl( div\, \mathcal X(z) + div\, \mathcal X(t)\bigr)- (N+2s)\frac{ \bigl(\mathcal X(z)-\mathcal X(t)\bigr)\cdot (z-t)}{|z-t|^2}\Bigr]\frac{1}{|z-t|^{N+2s}} .
\]
We fix $\xi\in \Omega$, and  choose the vector field $\mathcal X =id -\xi$. Hence  we
apply \eqref{poho_DFW} to the functions $G_s(x,\cdot)$ and $G_s(y,\cdot)$. Since the functions $G_s(x,\cdot)$ and $G_s(y,\cdot)$ are too irregular to be admissible in \eqref{poho_DFW},  we adapt the  arguments developed  in \cite{DjitteSueur} to approximate them by  $C^\infty$-functions of the form $\eta_k \psi_{\mu,x}G_s(x,\cdot)$ and $\eta_k \psi_{\gamma,y}G_s(y,\cdot)$, where $\eta_k$ and $\psi_{\mu,x}, \psi_{\gamma,y}$ are suitable cut-off functions which vanish near the boundary $\partial\Omega$, and near the singular points $x$ and $y$ of the Green functions, respectively. Then, we pass to the limit as $k\to+\infty$, $\mu\to 0^+$ and $\gamma\rightarrow 0^+$, adapting results from \cite{DjitteSueur, DjitteFallWethHadamard}, and thus we get \eqref{bilinearPohozaevGreenGENERAL}. The assumption $s>\frac{1}{2}$  guarantees the integrability of the function $z\mapsto \langle z-\xi,\nabla_zG_s(x,z)\rangle$, which is needed in the passages to the limit (see \emph{Step 3} in the proof of Theorem \ref{MainTheoremBilinear}). When $s \in (0,\frac{1}{2}]$, this integrability is recovered by choosing $\xi=x$ (see Remark \ref{remark:IntegrabilityNablaG}), this explains why we get  \eqref{bilinearPohozaevGreen} in this case.\\

To the best of our knowledge, there is currently no known counterpart of \eqref{bilinearPohozaevGreenGENERAL} in the local case 
$s=1$. We also establish such an analogous identity in the local setting.
\begin{theorem}
\label{Main_local_TheoremBilinear_s=1}
Let $N\in\mathbb N$, $N > 2$, and let $\Omega$ be a bounded open set of $\mathbb R^N$ of class $C^{1,1}$. Let $x,y\in\Omega$ with $x\neq y$. Then for every $\xi\in\mathbb R^N$ there holds
\begin{equation}\label{bilinearPohozaevGreenGENERAL_s=1}
\displaystyle\int_{\partial\Omega}\frac{ \partial G_1(x,\cdot)}{\partial\nu}\frac{\partial G_1(y,\cdot)}{\partial\nu} \langle \cdot-\xi,\nu\rangle \, d\sigma=(N-2) H_1(x,y)  +  \langle \nabla_x  H_1(y,x),x-\xi\rangle 
+  \langle \nabla_y  H_1(x,y),y-\xi\rangle,
\end{equation}
where $G_1$ and $H_1$ are the Green function and its regular part for the standard Laplacian in $\Omega$, respectively.
\end{theorem}
The proof of Theorem \ref{Main_local_TheoremBilinear_s=1} combines the previous strategy with some ideas from \cite{BrezisPeletier}. Indeed, it consists again in applying the integration-by-part formula, now the classical one, to suitable smooth approximations of  the classical Green function, but now   we can approximate  $G_1(x,\cdot)$ and $G_1(y,\cdot)$ with the smooth functions $u_{\rho,x}$ and $u_{\rho,y}$, which solve
\[
\left\{ 
\begin{array}{rcll} -\Delta u_{\rho,a}&=& \frac{1}{|B_{\rho}(a)|} \chi_{B_\rho(a)}  &\textrm{in }\Omega \\ u_{\rho,a}&=& 0&\textrm{on }\partial\Omega, \end{array}
\right.
\]
with $a=x$ and $a=y$, respectively, where $B_\rho(a)$ is the ball of radius $\rho>0$ centered at $a$, and $\chi$ denotes the characteristic function.  By this choice,  it follows that the fundamental solutions $F_1(x,\cdot)$ and $F_1(y,\cdot)$  are approximated by a function which can be computed explicitly. This easily allows to conclude, passing to the limit as $\rho\rightarrow 0^+$.  We stress that we can chose now $\rho$ to be the same in both the approximations. We also remark that, having $x\neq y$, simplifies the arguments if compared to the ones in \cite{BrezisPeletier}.

\section{Preliminaries}
Let $\delta_\Omega$ be the signed distance function $\delta_\Omega(z):=dist(z,\mathbb R^N\setminus\Omega)-dist(z,\Omega)$, and let us fix $\delta\in C^{1,1}(\R^N)$  such that it coincides with $\delta_\Omega$ near the boundary of $\Omega$, it is positive in $\Omega$ and negative in $\mathbb R^N\setminus\Omega$.\\
For $\eps >0$ we define the sets
\begin{equation}\label{eq:defSetOmegaeps}
\Omega_\eps:= \{z \in \R^N\::\: |\delta(z)|< \eps\} \qquad \text{and}\qquad 
\Omega_\eps^+:= \{z \in \R^N\::\: 0< \delta(z)< \eps\}= \{z \in \Omega\::\: \delta(z)< \eps\}.
\end{equation}
We also define $\Psi: \partial \Omega \times (-\eps, \eps) \to \Omega_{\eps}$ as the map 
\begin{equation}\label{def:localCoordBoundary}
\Psi(\sigma,r):= \sigma - r \nu(\sigma),
\end{equation}
where $\nu: \partial \Omega \to \R^N$ is the outward unit  normal vector field. Since  $\partial \Omega\in C^{1,1}$ by assumption, the map $\Psi$ is Lipschitz, moreover for  $\eps>0$ sufficiently small  it is  bi-Lipschitz.  In particular, $\Psi$ is a.e. differentiable, with $\Psi(\partial \Omega \times (0, \eps))= \Omega_{\eps}^+$,
$
\delta(\Psi(\sigma,r)) = r $, for $\sigma \in \partial \Omega,\: 0 \le r < \eps$,
Furthermore
  \[
 \|{\rm Jac}_{{\Psi}}\|_{L^\infty(\Omega_\eps)}< +\infty \quad \text{and}\quad {\rm Jac}_{{\Psi}}(\sigma,0)=1 \quad \text{for a.e. $\sigma \in \partial \Omega$.}
   \]
Following \cite{DjitteSueur,DjitteFallWethHadamard}, we define suitable cut-off functions, which vanish near the boundary $\partial \Omega$, and near a fixed point $x\in \Omega$, respectively.
Let $\rho\in C^\infty_c(-2,2)$, such that $0\leq \rho\leq 1$, and $\rho=1$ in $(-1,1)$.
For any $k\in \mathbb N^*$, we define $\eta_k: \mathbb R^N\to \mathbb R$
\begin{equation}\label{eta-k}
    \eta_k(z):=1-\rho(k\delta(z)).
\end{equation}
Furthermore, for $x\in \Omega$ fixed, and for any $\mu\in(0,1)$, we also define  $\psi_{\mu,x}:\R^N\to \mathbb R$ by 
\begin{equation}\label{psi-mu}
 \psi_{\mu,x}(z):=1-\rho\Big(\frac{8}{\delta^2_\Omega(x)}\frac{|x-z|^2}{\mu^2}\Big).
\end{equation}
In the rest of the section we collect convergence properties for integral functionals, involving both these cut-off functions and the Green functions $G_s(x,\cdot)$ and $G_s(y,\cdot)$, as $k\rightarrow +\infty$ or $\mu\rightarrow 0^+$. 
 Lemma \ref{lemma:similarFHFA} is new, its proof is inspired by arguments in \cite{DjitteFallWethHadamard}, and relies on estimates already obtained there. Lemma \ref{lemma:lm_2.4_DS_revisitedFIRST_PART} and Lemma \ref{lemma:lm_2.4_DS_revisited} are slight modifications of a previous results in \cite{DjitteSueur}. Lemma \ref{lemma_2.3_Djitte_Sueur_corrected},  \ref{lemma:lm2.2_DS_revisited} and \ref{psi_lm} are instead taken from  \cite{DjitteSueur,DjitteFallWethHadamard},  we repeat them here for the reader's convenience.\\\\
Let us recall the fractional product rule: 
 \begin{equation}\label{pl}
(-\Delta)^s (uv) =  v(-\Delta)^s u + u (-\Delta)^s v  -\mathcal I_s [u,v],
\end{equation}
where
\begin{align}\label{def-I-s}
\mathcal I_s [u,v](\cdot) :=  c_{N,s} \, \pv\int_{\R^N} \frac{\big(u(\cdot)-u(t)\big) \big(v(\cdot)-v(t)\big)}{|\cdot-\,t\,|^{N+2s}}dt .
\end{align}
From \cite[Proposition 2.4]{DjitteFallWethHadamard} (see also Remark 2.5) applied to $u=v+w$ and $u=v-w$ we get
\begin{lemma}
  \label{lemma_2.3_Djitte_Sueur_corrected}
Let $v,w\in C^s(\mathbb R^N) \cap C^{1}_{loc}(\Omega)$, $v\equiv w\equiv 0$ in $\mathbb R^N\setminus\Omega$, satisfying, for some $\alpha\in (0,1)$, the following regularity properties:
\[
\frac{v}{\delta^s_{\Omega}}, \frac{w}{\delta^s_{\Omega}} \in C^{\alpha}(\overline\Omega) \qquad \text{and}\qquad 
 \delta_{\Omega}^{1-\alpha} \nabla \frac{v}{\delta^s_{\Omega}},\, \delta_{\Omega}^{1-\alpha} \nabla \frac{w}{\delta^s_{\Omega}} \quad \text{are bounded in a neighborhood of $\partial \Omega$.}
\]
Then for all $\mathcal X \in C^{0,1}\left(\mathbb{R}^{N}, \mathbb{R}^{N}\right)$, we have
\begin{equation}
  \label{eq:limit-identity}
\lim_{k\to +\infty}\Bigg(
\int_{\Omega}\langle\nabla (\eta_kw), \mathcal X\rangle \Bigl( v   (-\Delta)^s \eta_k  -\mathcal I_s[v,\eta_k]\Bigr)+\int_{\Omega}\langle\nabla (\eta_kv), \mathcal X\rangle \Bigl( w(-\Delta)^s \eta_k  -\mathcal I_s[w,\eta_k]\Bigr)\,dz\Bigg)
= \Gamma^2(1+s)  \int_{\partial\Omega} \frac{w}{\delta^s_{\Omega}}\frac{v}{\delta^s_{\Omega}}\, \langle \mathcal X, \nu\rangle \, d\sigma.
\end{equation}
where $\nu$ is the outward unit normal to $\partial\Omega$.
\end{lemma}
\begin{remark}
    We stress that here, differently from \cite{DjitteFallWethHadamard}, $\nu$ is the outward unit normal to $\partial\Omega$, this explains the different sign on the right hand side of \eqref{eq:limit-identity}.
\end{remark}

\begin{lemma} \label{lemma:similarFHFA} Let $v\in C^s(\mathbb R^N)$, $v\equiv 0$ in $\R^N\setminus\Omega$, satisfying  
$\frac{v}{\delta^s_{\Omega}}\in C^{\alpha}(\overline\Omega)$ for some $\alpha\in (0,1)$.
Let $w\in C(\overline\Omega)$. Then
   \[ \lim_{k\rightarrow +\infty}\int_{\Omega} w \eta_k\Bigl( v(-\Delta)^s \eta_k  -\mathcal I_s[v,\eta_k]\Bigr)dz =0\]
\end{lemma}
\begin{proof}
    Taking inspiration from the proof of \cite[Proposition 2.4]{DjitteFallWethHadamard}, for $k \in \N$, we define $q_k:\Omega \to \R$
\begin{equation}
  \label{eq:def-ge-k}
q_k:= w  \eta_k\Bigl( v   (-\Delta)^s \eta_k  -\mathcal I_s[v,\eta_k]\Bigr),
\end{equation}
and, for $\eps>0$,  
we split the integral as
\begin{equation}\label{splitInt}\int_{\Omega} q_k\,dz = \int_{\Omega\setminus \Omega^\eps} q_k\,dz +  \int_{\Omega^{\eps}_+} q_k\,dz, \end{equation}
where $\Omega_\eps$ and $\Omega_\eps^+$ are the sets defined in \eqref{eq:defSetOmegaeps}.
We show that, for every $\eps>0$, we  have 
\begin{equation}
  \label{eq:reduction-to-eps-collar-prelim-1}
\lim_{k\to +\infty}\int_{\Omega \setminus \Omega_\eps} q_k\,dz =0,
\end{equation}
and that there exists $\eps'>0$ such that for any $\eps\in (0,\eps')$, we also have
\begin{equation}
\label{eq:SecondtermgoestoZERO}
\lim_{k\to +\infty}\int_{\Omega_\eps^+} q_k\,dz =0.
\end{equation}
The conclusion than follows from \eqref{eq:SecondtermgoestoZERO},  \eqref{eq:reduction-to-eps-collar-prelim-1} and \eqref{splitInt}.\\
\\
The proof of \eqref{eq:reduction-to-eps-collar-prelim-1} is the same as in  \cite[Proof of (6.5)]{DjitteFallWethHadamard}, we rewrite it for completeness. It consists in showing that 
 $q_k \to  0$ pointwise on $\Omega \setminus \Omega_\eps$, that
 $\|q_k\|_{L^\infty(\Omega \setminus \Omega_\eps)}$ is bounded independently of $k$, so that \eqref{eq:reduction-to-eps-collar-prelim-1} follows by the dominated convergence theorem.
Since $|\eta_k|\leq 1$, $\eta_k  \to 1$ pointwise on $\R^N \setminus \partial \Omega$, and therefore a.e. on $\R^N$, it is sufficient to check that, as $k\rightarrow +\infty$
\begin{align}
&(-\Delta)^{s} \eta_{k}\rightarrow 0,\label{LapEtaGoes0}
\\\label{IEtaGoes0}
&\mathcal{I}_{s}\left[v,\eta_k \right]\rightarrow 0,
\end{align}
pointwise on $\Omega\setminus\Omega_{\eps}$, and that
\begin{align}\label{Fraceta_Bounded_unif_in_k}
&\|(-\Delta)^s \eta_k\|_{L^\infty(\Omega \setminus \Omega_\eps)}\leq C,\\\label{Ieta_Bounded_unif_in_k}
&\|\mathcal I_s[v,\eta_k]\|_{L^\infty(\Omega \setminus \Omega_\eps)}\leq C,\end{align} 
where $C>0$ is independent of $k$.
Choosing a compact neighborhood $K \subset \Omega$ of $\Omega \setminus \Omega_\eps$, we have 
$$
(-\Delta)^s \eta_k(z)= c_{N,s}\int_{\R^N \setminus K}\frac{1- \eta_k(t)}{|z-t|^{N+2s}}dt \qquad \text{for $z \in \Omega \setminus \Omega_\eps$ and $k$ sufficiently large,}
$$
and $$\frac{|1- \eta_k(t)|}{|z-t|^{N+2s}}\leq \frac{C}{1+|t|^{N+2s}}\qquad\text{ for $z \in \Omega \setminus \Omega_\eps,\: t\in \R^N \setminus K$},$$ where $C>0$ independent of $z$ and $t$. Consequently, \eqref{Ieta_Bounded_unif_in_k} holds true, and  by the dominated convergence theorem also \eqref{LapEtaGoes0} is satisfied. Similarly, one gets \eqref{IEtaGoes0} and \eqref{Ieta_Bounded_unif_in_k}.\\\\
Next we show \eqref{eq:SecondtermgoestoZERO}. Let us fix $\eps>0$ small so that the map $\Psi: \partial \Omega \times (-\eps, \eps) \to \Omega_{\eps}$ defined  in 
\eqref{def:localCoordBoundary} is bi-Lipschitz.
For $0<\eps' \le \eps$, 
\begin{equation}
 \lim_{k\to +\infty}\int_{\Omega_{\eps'}^+} q_k\,dz = \lim_{k\to +\infty}\int_{\partial \Omega} \int_{0}^{\eps'}{\rm Jac}_{{\Psi}}(\sigma,r) q_k({\Psi}(\sigma,r))\,dr d\sigma = \lim_{k\to +\infty}\frac{1}{k} \int_{\partial \Omega} \int_{0}^{k \eps'}j_k(\sigma,r) Q_k(\sigma,r) \,dr d\sigma, \label{claim-first-reduction}
\end{equation}
where   
\begin{equation}
  \label{eq:def-j-k-G-k}
j_k(\sigma,r) := {\rm Jac}_{{\Psi}}(\sigma,\frac{r}{k}) \quad \text{and}\quad Q_k(\sigma,r):=q_k({\Psi}(\sigma,\frac{r}{k})) \qquad \text{for a.e. $\sigma \in \partial \Omega,\: 0 \le r < k \eps.$}
\end{equation}
We note that 
  \begin{align*}
&\|j_k\|_{L^\infty(\partial \Omega \times [0,k\eps))} \le \|{\rm Jac}_{{\Psi}}\|_{L^\infty(\Omega_\eps)}< +\infty \quad \text{for all $k$, and}\\
&\lim_{k \to +\infty} j_k(\sigma,r)= {\rm Jac}_{{\Psi}}(\sigma,0)=1 \quad \text{for a.e. $\sigma \in \partial \Omega$, $r>0$.}
   \end{align*} 
By definition of the functions $q_k$ in (\ref{eq:def-ge-k}), 
\[
Q_k(\sigma,r) =Q_k^0(\sigma,r)\Big(Q_{k}^1(\sigma,r) - Q_{k}^2(\sigma,r)\Big) \qquad \text{for $\sigma \in \partial \Omega,\: 0 \le r < k \eps$}
\]
with 
\begin{align*}
Q_k^0(\sigma,r) &:= w({\Psi}(\sigma,\frac{r}{k}))\eta_k({\Psi}(\sigma,\frac{r}{k}))\\
Q_k^1(\sigma,r) &:= \Bigl(v   (-\Delta)^s \eta_k\Bigr)({\Psi}(\sigma,\frac{r}{k})) \\
Q_k^2(\sigma,r) &:= \mathcal I_s[v,\eta_k]({\Psi}(\sigma,\frac{r}{k})).
\end{align*}  
Let us observe that \[
\big|Q_0^1(\sigma,r)\big|\leq \|w\|_{L^{\infty}(\Omega)}.\]
Furthermore, from \cite[Corollary 6.4]{DjitteFallWethHadamard} and  \cite[Lemma 6.8]{DjitteFallWethHadamard} we know that there exists $\eps'>0$ with the property that    
\[
\big|k^{-s}Q_k^1(\sigma,r)\big|\le \frac{C r^s}{1+r^{1+2s}} \quad\mbox{ and }\quad
\big|k^{-s} Q_k^2(\sigma,r)\big| \le \frac{C}{1+r^{1+s}},
\]
for $k \in \N$, $0 \le r < k \eps'$, $\sigma \in \partial \Omega$, with a constant $C>0$, independent of $k$ and $(\sigma,r)$.
Hence, we deduce
\begin{equation}
  \label{eq:complete-proof-bound}
\frac{\big|Q_k(\sigma,r)\big|}{k} \leq C \|w\|_{L^{\infty}(\Omega)}\frac{1}{k^{1-s}}
\Bigl(\frac{r^{s}}{1+r^{1+2s}}+\frac{1}{1+r^{1+s}}\Bigr),
\end{equation}
By \eqref{claim-first-reduction}, and a  direct computation,  
\begin{align*}
\lim_{k \to +\infty}\Bigl|\,\int_{\Omega}q_k\, dz \,\Bigr|&\leq \lim_{k\to +\infty}\frac{1}{k} \int_{\partial \Omega} \int_{0}^{k \eps'}j_k(\sigma,r) \big|Q_k(\sigma,r)\big| \,dr d\sigma\\
&\leq C \|w\|_{L^{\infty}(\Omega)}\lim_{k\to +\infty}\frac{1}{k^{1-s}}
\int_{\partial \Omega} \int_{0}^{k \eps'}j_k(\sigma,r) 
\Bigl(\frac{r^{s}}{1+r^{1+2s}}+\frac{1}{1+r^{1+s}}\Bigr)\,dr d\sigma\\
&= C \|w\|_{L^{\infty}(\Omega)}
|\partial \Omega| \widetilde C_s 
\left(\lim_{k\to +\infty}\frac{1}{k^{1-s}}\right)=0
 \end{align*}
since $s \in (0,1)$, where $\widetilde C_s:=\int_{0}^{+\infty}
\Bigl(\frac{r^{s}}{1+r^{1+2s}}+\frac{1}{1+r^{1+s}}\Bigr)\,dr<+\infty$ since the integrand is integrable over $[0,+\infty)$. 
\end{proof}

\begin{lemma}\label{lemma:lm_2.4_DS_revisitedFIRST_PART}
Let $v$ such that $|v|\leq C\delta^s$. Let $w\in L^{\infty}(\Omega)$. Then for all $\mathcal X \in C^{0,1}\left(\mathbb{R}^{N}, \mathbb{R}^{N}\right)$, we have
\[
\lim_{k\to+\infty} \int_{\Omega}\eta_kv\langle\nabla\eta_k ,\mathcal X\rangle w\,dz=0.
\]
\end{lemma}
\begin{proof}
Recalling  the definition of $\eta_k$
\begin{align*}
\Big|\int_{\Omega}\eta_kv\langle\nabla\eta_k ,\mathcal X\rangle w\,dz\Big|
&\leq \|w\|_{L^{\infty}(\Omega)}\int_{\Omega}|v|\, \big|\langle\nabla\eta_k ,\mathcal X\rangle\big| dz= \|w\|_{L^{\infty}(\Omega)}k\int_{\Omega^+_\frac{2}{k}\setminus\Omega^+_{\frac{1}{k}}}|v|\, |\rho'(k\delta)|\, \big|\langle\nabla\delta ,\mathcal X\rangle\big| dz
\\
&\leq \|w\|_{L^{\infty}(\Omega)}\widetilde C k\int_{\Omega^+_\frac{2}{k}\setminus\Omega^+_{\frac{1}{k}}}|v| dz
\leq 2C\|w\|_{L^{\infty}(\Omega)}\widetilde C\int_{\Omega^+_\frac{2}{k}\setminus\Omega^+_{\frac{1}{k}}}\frac{|v|}{\delta}dz
\\
&\leq 2\|w\|_{L^{\infty}(\Omega)}\widetilde C C \int_{\Omega^+_\frac{2}{k}\setminus\Omega^+_{\frac{1}{k}}}\delta^{s-1}\,dz,
\end{align*}
where we have used that $1\leq k\delta(z)\leq 2$ for $z\in\Omega^+_\frac{2}{k}\setminus\Omega^+_\frac{1}{k}$, with $\Omega^+_{\alpha}:=\{z\in\Omega\ :\ \delta(z)<\alpha\},$ and  the assumption $|v|\leq C\delta^s$.
In order to compute the last integral, we fix $\eps>0$ small enough, such that the map $\Psi: \partial \Omega \times (-\eps, \eps) \to \Omega_{\eps}$,  
$
\Psi(\sigma,r):= \sigma - r \nu(\sigma),
$ defined in \eqref{def:localCoordBoundary}
is bi-Lipschitz. Lt us recall that $\delta(\Psi(\sigma,r)) = r $ for $(\sigma,r)\in \partial\Omega\times (0,\eps)$, and that $\|{\rm Jac}_{{\Psi}}\|_{L^\infty(\Omega_\eps)}< +\infty$. Hence, since $\Omega^+_\frac{2}{k}\setminus\Omega^+_{\frac{1}{k}}\subset \Omega_\eps^+ $ for $k$ large, we have 
\begin{align*}
&\int_{\Omega^+_\frac{2}{k}\setminus\Omega^+_{\frac{1}{k}}}\delta^{s-1}\,dz
=\int_{\partial\Omega}\int_{\frac{1}{k}}^{\frac{2}{k}}
{\rm Jac}_{{\Psi}}(\sigma,r)\delta(\Psi(\sigma,r))^{s-1}drd\sigma
=\int_{\partial\Omega}\int_{\frac{1}{k}}^{\frac{2}{k}}
{\rm Jac}_{{\Psi}}(\sigma,r)r^{s-1}drd\sigma
 \\
 &\leq \|{\rm Jac}_{{\Psi}}\|_{L^\infty(\Omega_\eps)}|\partial\Omega|\lim_{k\rightarrow +\infty} \int_{\frac{1}{k}}^{\frac{2}{k}} r^{s-1}dr=\|{\rm Jac}_{{\Psi}}\|_{L^\infty(\Omega_\eps)}|\partial\Omega|\frac{2^s-1}{s}\frac{1}{k^s}\leq C\frac{1}{k^s}.\end{align*}
The conclusion follows passing to the limit as $k\rightarrow +\infty$.
\end{proof}

\begin{remark}[Estimates on $G_s(x,z)$ and  $\nabla_z G_s(x,z)$]\label{remark:IntegrabilityNablaG}$\,$\\
Let $x\in\Omega\subset\mathbb R^N$, clearly $z\mapsto G_s(x,z)\in C^1_{loc}(\Omega\setminus\{x\}$).
It was proved by Chen and Song \cite{ChenSong}, and Kulczycki \cite{Kulczycki} that, if $\Omega\subset\mathbb R^N$ is a bounded open set  of class $C^{1,1}$, then for all $x,z\in\Omega$ with $x\neq z$, there holds 
\begin{equation}\label{greenfunctEstimate}
    C_1\min\Big(\frac{1}{|x-z|^{N-2s}}, \frac{\delta^s_\Omega(x)\delta^s_\Omega(z)}{|x-z|^N}\Big)\leq \frac{G_s(x,z)}{\Lambda_{N,s}}\leq \min\Big(\frac{1}{|x-z|^{N-2s}}, C_2\frac{\delta^s_\Omega(x)\delta^s_\Omega(z)}{|x-z|^N}\Big),
\end{equation}
for some constant $C_1, C_2>0$ and an explicit constant $\Lambda_{N,s}$. As a consequence, since in particular $G_s(x,z)\leq \frac{\Lambda_{N,s}}{|x-z|^{N-2s}}\in L^1(\Omega)$, it follows that
\begin{equation}
z\mapsto G_s(x,z)\in L^1(\Omega).
\end{equation}
Furthermore, from \cite[Corollary 3.3]{BogdanKulczyckiNowak} it is known that, for a bounded open set $\Omega\subset\mathbb R^N$,   for all $x,z\in\Omega$ with $x\neq z$, there holds
\begin{equation}\label{greenfunctGRADIENTEstimate}
    \big|\nabla _z G_s(x,z)\big|
    \leq 
    \left\{\begin{array}{ll}
    N\displaystyle{\frac{G_s(x,z)}{|x-z|} }&\mbox{if }z\in N_x \\[0.4cm]
   N\displaystyle{ \frac{G_s(x,z)}{\delta_\Omega(z)}}& \mbox{if }z\in \Omega\setminus N_x
    \end{array}
    \right..
\end{equation}
where
\[
N_x:=\{z\in\Omega\ :\ |x-z|
\leq\delta_\Omega(z)\}.
\]
When $\Omega$ is of class $C^{1,1}$, and $x\in\Omega$, it follows from \eqref{greenfunctEstimate} and \eqref{greenfunctGRADIENTEstimate} that 
\begin{align}\label{eq:GradGL1}
z\mapsto  \big|\nabla_z G_s(x,z)\big|\in L^1(\Omega), \ &\mbox{if}\ s>\frac{1}{2}
\\
\label{eq:FieldGradGL^1}
z\mapsto \langle z-x,\nabla_z G_s(x,z)\rangle \in L^1(\Omega),\ & \mbox{for any}\ s\in (0,1).
\end{align}
Indeed, let us fix $\eps>0$ small enough so that $\eps<\frac{\delta(x)}{2}$, and that $|x-z|>\frac{\delta(x)}{2}$ for every $z\in\Omega_\eps^+$, where  $\Omega_\eps^+$ is the inner neighborhood of $\partial\Omega$ as defined in \eqref{eq:defSetOmegaeps}. 
By this choice, for any $z\in\Omega_\eps^+$ one has $|x-z|>\frac{\delta(x)}{2}>\eps>\delta(z)$, namely 
\[\Omega_\eps^+\subset (\Omega\setminus N_x).\]
Let us split $\Omega= N_x\cup \Omega_{\eps}^+\cup I_{x,\eps}$,
where
\[
I_{x,\eps}:=\Omega\setminus (N_x\cup \Omega_\eps^+)\ (\subset \Omega\setminus N_x).\]
By \eqref{greenfunctGRADIENTEstimate} and \eqref{greenfunctEstimate} it follows that:
\begin{align}\label{EstimateN}
&\mbox{for  }z\in N_x: \quad\big|\nabla_z G_s(x,z)\big|\leq N  \frac{G_s(x,z)}{|x-z|}\leq N \Lambda_{N,s}|x-z|^{-(N-2s+1)}\\\label{EstimateD}
&\mbox{for  }z\in \Omega_{\eps}^+: 
\quad\big|\nabla_z G_s(x,z)\big|\leq N  \frac{G_s(x,z)}{\delta_{\Omega}(z)}\leq \frac{N \Lambda_{N,s}C_2 \delta^s_\Omega(x)}{|x-z|^N}\delta_{\Omega}^{s-1}(z)
\\\label{EstimateE}
&\mbox{for  }z\in I_{x,\eps}:\quad \big|\nabla_z G_s(x,z)\big|\leq N  \frac{G_s(x,z)}{\delta_{\Omega}(z)}\leq \frac{N \Lambda_{N,s}C_2 \delta^s_\Omega(x)}{|x-z|^N\delta_{\Omega}^{1-s}(z)}.
\end{align}
As a consequence
\begin{align*}
\int_{\Omega}\big|\nabla_z G_s(x,z)\big|\,dz
\leq &C\int_{\Omega}|x-z|^{-(N-2s+1)}\,dz
+ 
C\int_{\Omega_\eps^+}
\delta_{\Omega}^{s-1}(z)
\,dz
+ C |\Omega|,\\
\int_{\Omega}\big|\langle z-x,\nabla_z G_s(x,z)\rangle\big|\,dz\leq & C\int_{\Omega}|x-z|^{-(N-2s)}\,dz
+ 
C\int_{\Omega_\eps^+}
\delta_{\Omega}^{s-1}(z)
\,dz
+ C |\Omega|,
\end{align*}
where $C=C(N,s,x)>0$ and, in order to estimate the term in \eqref{EstimateD} we have used that $|x-z|>\frac{\delta(x)}{2}$ for every $z\in\Omega_\eps$, 
while for the term in \eqref{EstimateE} we have used that $|x-z|>\delta_{\Omega}(z)>\eps$.
The conclusion follows observing that 
\begin{align}\nonumber
& \int_{\Omega}|x-z|^{-(N-2s+1)}\,dz    <+\infty, \quad\mbox{if}\ s\in(\frac{1}{2},1);\\\nonumber
&\int_{\Omega}|x-z|^{-(N-2s)}\,dz <+\infty,\quad\mbox{if}\ s\in(0,1);\\\label{computationInLocalCoord}
&\int_{\Omega_\eps^+}
\delta_{\Omega}^{s-1}(z)
\,dz =\int_{\partial\Omega}\int_{0}^{\eps}
{\rm Jac}_{{\Psi}}(\delta\circ \Psi)^{s-1}drd\sigma 
\leq C \int_{0}^{\eps} r^{s-1}dr<+\infty,
\end{align}
where the map  $\Psi: \partial \Omega \times (-\eps, \eps) \to \Omega_{\eps}$ in \eqref{computationInLocalCoord} is defined in \eqref{def:localCoordBoundary}, and it is bi-Lipschitz for $\eps$ small enough, moreover we have used that $\delta(\Psi(\sigma,r)) = r $ for $(\sigma,r)\in \partial\Omega\times (0,\eps)$, and that $\|{\rm Jac}_{{\Psi}}\|_{L^\infty(\Omega_\eps)}< +\infty$.  
Actually, we stress that in \cite[Lemma 9]{BogdanJakubowski}, \eqref{eq:GradGL1} and \eqref{eq:FieldGradGL^1} have been shown to hold uniformly in $x$, but both only when $s>\frac{1}{2}$).
\end{remark}

\begin{remark}[Regularity of $\psi_{\mu,x}G_s(x,\cdot)$]
In \cite[Lemma 2.2]{DjitteSueur} it has been shown that
\begin{equation}
    \label{fracLappsiGBounded}
    (-\Delta)^s\Big(\psi_{\mu,x}G_s(x,\cdot)\Big)\in L^{\infty}(\Omega),
\end{equation} 
furthermore, $\psi_{\mu,x}G_s(x,\cdot)=0$ in $\mathbb R^N\setminus \Omega$. As a consequence, by interior and boundary regularity results (see \cite{Ros-Oton_Serra_Regularity_2016,Ros-Oton_Serra_Dirichlet_2014, Fall_Jarohs_2021}), the function $v:=\psi_{\mu,x}G_s(x,\cdot)$, satisfies:  
\begin{equation}\label{Regularity_PsiG}
v\in C^s(\mathbb R^N)\cap C^1_{loc}(\Omega),\quad \frac{v}{\delta^s}\in C^{\alpha}(\overline\Omega)\ \forall\alpha\in (0,s), \quad \delta^{1-\alpha} \big|\nabla \frac{v}{\delta^s}\big|\leq C\mbox{ in }\Omega, \quad \delta^{1-s}\nabla v\in C^{\beta}(\overline\Omega)\mbox{ for some }\beta \in (0, \alpha).\end{equation}
Let us observe that, the latter property, it follows  that
\begin{equation}\label{eq:nablaPsiGL1}
|\nabla v|\in L^1(\Omega),\  \mbox{for any}\ s\in (0,1).
\end{equation}
Indeed, let us fix $\eps>0$ small enough, and split the integral
\begin{align*}
\int_{\Omega}|\nabla v|\,dz\leq C\int_{\Omega}\delta^{s-1}\,dz=\int_{\Omega\setminus\Omega_{\eps}}\delta^{s-1}\,dz+\int_{\Omega_{\eps}^+}\delta^{s-1}\,dz
\leq C + \int_{\Omega_{\eps}^+}\delta^{s-1}\,dz <+\infty,
\end{align*}
where $\Omega_{\eps}^+$ is as in \eqref{eq:defSetOmegaeps}, and
the conclusion follows as in \eqref{computationInLocalCoord}. 
\end{remark}
We  recall the following results from \cite{DjitteSueur}:
\begin{lemma} \label{lemma:lm2.2_DS_revisited} Let $x \in \Omega$, $G_{s}(x, \cdot)$  the Green function with singularity at $x$, and $\psi_{x, \mu}$ the cut-off function defined as in \eqref{psi-mu}. Let $\varepsilon>0$ such that $\overline{B_{2 \varepsilon}(x)} \subset \Omega$.
If $w\in L^1(\Omega\setminus B_{\varepsilon}(x))$, then   we have
\begin{align*}
& \lim _{\mu \rightarrow 0^{+}} \int_{\Omega\setminus B_{\varepsilon}(x)} w\, G_{s}(x, \cdot)(-\Delta)^{s} \psi_{\mu, x} dz=0;
\\
& \lim _{\mu \rightarrow 0^{+}} \int_{\Omega\setminus B_{\varepsilon}(x)} w\, \mathcal{I}_{s}\left[G_{s}(x, \cdot), \psi_{\mu, x}\right] dz=0.
\end{align*}
\end{lemma}
\begin{proof}
    This is essentially contained in the proof of \cite[Lemma 2.2]{DjitteSueur}, where it is proved that 
    $G_s(x,z)(-\Delta)^{s} \psi_{\mu, x}(z)\rightarrow 0$ and $\mathcal{I}_{s}\left[G_{s}(x, \cdot), \psi_{\mu, x}\right](z)\rightarrow 0
   $, as $\mu\rightarrow 0^+$ and that $\left|G_s(x,z)(-\Delta)^{s} \psi_{\mu, x}(z)\right|\leq C(\varepsilon)$ and $\left|\mathcal{I}_{s}\left[G_{s}(x, \cdot), \psi_{\mu, x}\right](z)\right|\leq C(\varepsilon)$, for $\mu>0$ small and $z\in\Omega\setminus B_{\varepsilon}(x)$.  So that the conclusion follows by the dominated convergence theorem.
\end{proof}

\begin{lemma} (\cite[Lemma 2.6]{DjitteSueur})
\label{psi_lm}
 Fix $x \in \Omega$, and let $G_{s}(x, \cdot)$ be the Green function with singularity at $x$, and $\psi_{x, \mu}$ be the cut-off function defined as in \eqref{psi-mu}. Let $\varepsilon>0$ be such that $\overline{B_{2 \varepsilon}(x)} \subset \Omega$. If  $w\in C(\overline{B_{\varepsilon}(x)})$\footnote{In \cite[Lemma 2.6]{DjitteSueur} the continuity of $w$ is assumed in the whole $\Omega$, but the proof uses only the continuity in $x$, see \cite[Remark 2.7]{DjitteSueur}}
 then we have
\begin{align}
& \lim _{\mu \rightarrow 0^{+}} \int_{B_{\varepsilon}(x)} w\, G_{s}(x, \cdot)(-\Delta)^{s} \psi_{\mu, x} d z=-w(x)\\
& \lim _{\mu \rightarrow 0^{+}} \int_{B_{\varepsilon}(x)} w\, \mathcal{I}_{s}\left[G_{s}(x, \cdot), \psi_{\mu, x}\right] d z=-2 w(x)
\end{align}
\end{lemma}

\begin{lemma}\label{lemma:lm_2.4_DS_revisited}
Let $v\in C^1_{loc}(\Omega)$, $|v|\leq C\delta^s$, $|\nabla v|\in L^1(\Omega)$.  Fix $x \in \Omega$, and let $G_{s}(x, \cdot)$ be the Green function with singularity at $x$, and $\psi_{x, \mu}$ be the cut-off function defined as in \eqref{psi-mu}. Let $\eta_k$ be the cut-off function as defined in \eqref{eta-k}.
Then, for all $\mathcal X \in C^{0,1}\left(\mathbb{R}^{N}, \mathbb{R}^{N}\right)$, we have 
  \[ 
\lim_{\mu\to 0^+}\lim_{k\to+\infty} 
\int_{\Omega}\eta_k\langle\nabla\big(\eta_k v\big),\mathcal X\rangle(-\Delta)^s\Big(\psi_{\mu,x}G_s(x,\cdot)\Big)dz 
=
\langle\nabla v (x), \mathcal X (x)\rangle.
\]
\begin{proof}
It is similar to the proof of \cite[Lemma 2.4]{DjitteSueur}. We first apply the Leibniz rule to get
\begin{align}\nonumber
\int_{\Omega}\eta_k\langle\nabla\big(\eta_k v\big),\mathcal X\rangle(-\Delta)^s\Big(\psi_{\mu,x}G_s(x,\cdot)\Big)dz
&= 
\int_{\Omega}\eta_kv\langle\nabla\eta_k ,\mathcal X\rangle(-\Delta)^s\Big(\psi_{\mu,x}G_s(x,\cdot)\Big)dz\\
&+ \int_{\Omega}\eta_k^2\langle\nabla v,\mathcal X\rangle
(-\Delta)^s\Big(\psi_{\mu,x}G_s(x,\cdot)\Big)dz.\label{Scntrm}
\end{align}
For the  first term in \eqref{Scntrm},  since $|v|\leq C\delta^s$ by assumption, and $w:=(-\Delta)^s\Big(\psi_{\mu,x}G_s(x,\cdot)\Big)\in L^{\infty}(\Omega)$ by \eqref{fracLappsiGBounded}, we can use Lemma \ref{lemma:lm_2.4_DS_revisitedFIRST_PART},  hence $\lim_{k\to+\infty} \int_{\Omega}\eta_kv\langle\nabla\eta_k ,\mathcal X\rangle(-\Delta)^s\Big(\psi_{\mu,x}G_s(x,\cdot)\Big)\ dz=0$, and so
\begin{equation}\label{firstTerminScnAA}\lim_{\mu\to 0^+}\lim_{k\to+\infty} \int_{\Omega}\eta_kv\langle\nabla\eta_k ,\mathcal X\rangle(-\Delta)^s\Big(\psi_{\mu,x}G_s(x,\cdot)\Big)\,dz=0.\end{equation} For the second term in \eqref{Scntrm}, we first pass to the limit as $k\rightarrow +\infty$, by the dominated convergence theorem (using the assumption $|\nabla v|\in L^1(\Omega)$, and  that also $(-\Delta)^s\Big(\psi_{\mu,x}G_s(x,\cdot)\Big)\in L^{\infty}(\Omega)$), and then we use the product rule  \eqref{pl}, observing that $\psi_{\mu,x}(-\Delta)^s G_s(x,\cdot)=0$:
\begin{align}\nonumber\lim_{k\rightarrow +\infty}\int_{\Omega}\eta_k^2\langle\nabla v,\mathcal X\rangle
(-\Delta)^s\Big(\psi_{\mu,x}G_s(x,\cdot)\Big)dz&=
\int_{\Omega}\langle\nabla v,\mathcal X\rangle
(-\Delta)^s\Big(\psi_{\mu,x}G_s(x,\cdot)\Big)dz\\\label{interm}
&=\int_{\Omega}\langle\nabla v,\mathcal X\rangle \Bigg(G_s(x,\cdot)(-\Delta)^s\psi_{\mu,x}- \mathcal I_s[G_s(x,\cdot), \psi_{\mu,x}]\Bigg)dz.\end{align}
Finally we pass to the limit as $\mu\rightarrow 0^+$ into \eqref{interm}, applying both Lemma \ref{lemma:lm2.2_DS_revisited} and Lemma \ref{psi_lm} with the function $w=\langle\nabla v , \mathcal X\rangle$ (which is $L^1(\Omega)\cap C_{loc}(\Omega)$ by assumption), thus getting
\begin{align}
\label{secondTerminScnAA}\nonumber
\lim_{\mu\to 0^+}\lim_{k\rightarrow +\infty}\int_{\Omega}\eta_k^2\langle\nabla v,\mathcal X\rangle
(-\Delta)^s\Big(\psi_{\mu,x}G_s(x,\cdot)\Big)dz&=
\lim_{\mu\to 0^+}\int_{\Omega}\langle\nabla v,\mathcal X\rangle
\Bigg(G_s(x,\cdot)(-\Delta)^s\psi_{\mu,x}- \mathcal I_s[G_s(x,\cdot), \psi_{\mu,x}]\Bigg)dz\\
&=\langle\nabla v (x), \mathcal X (x)\rangle.\end{align}
Substituting \eqref{secondTerminScnAA} and \eqref{firstTerminScnAA} into \eqref{Scntrm} we get the conclusion.
\end{proof}
\end{lemma}

From \cite[Theorem 1.3]{DjitteFallWeth}, choosing the vector field $\mathcal X:=id -\xi$, we have
\begin{lemma} Let $u,v\in \mathcal H^s_0(\Omega)$. Assume that $(-\Delta)^su,\,(-\Delta)^sv\in L^\infty(\Omega)$ if $2s>1$ and $(-\Delta)^su,\,(-\Delta)^sv\in C^\alpha_{loc}(\Omega)\cap L^\infty(\Omega)$ with $\alpha>1-2s$ if $2s\leq 1$.  Then for every $\xi\in\mathbb R^N$ there holds
\begin{align}\label{eq:pohozaev.1}
\int_\Omega\langle \cdot-\xi,\nabla u\rangle(-\Delta)^{s}v\,dx+\int_\Omega \langle \cdot-\xi,\nabla v\rangle(-\Delta)^{s}u\,dx=(2s-N)\int_{\Omega}u(-\Delta)^{s}vdx-\Gamma(1+s)^{2}\int_{\partial\Omega}\frac{u}{\delta^{s}}\frac{v}{\delta^{s}}\langle \cdot-\xi,\nu\rangle\,d\sigma.
\end{align}
where $\nu$ is the outward unit normal to $\partial\Omega$ and $\delta:=dist (\ \cdot\ , \mathbb R^N\setminus\Omega)$.
\end{lemma}

\section{The proof of Theorem \ref{MainTheoremBilinear}}
We apply the Pohozaev identity \eqref{eq:pohozaev.1} to the regular functions $u(z)= \psi_{\gamma,y}(z)G_s(y,z)\eta_k(z)$ and $v(z)=\psi_{\mu,x}(z)G_s(x,z)\eta_k(z)$ with $x\neq y$ to get, for any fixed $\xi\in\mathbb R^N$ 
\begin{align} 
\nonumber
(2s-N)\int_{\Omega}\psi_{\gamma,y}G_s(y,\cdot\,)\eta_k(-\Delta)^{s}\Big(\psi_{\mu,x}G_s(x,\cdot\,)\eta_k\Big)\,dz
&=\int_\Omega\langle \,\cdot -\xi,\nabla \Big(\psi_{\gamma,y}G_s(y,\cdot\,)\eta_k\Big)\rangle(-\Delta)^{s}\Big(\psi_{\mu,x}G_s(x, \cdot\, )\eta_k\Big)\,dz\\
& + \int_\Omega \langle \,\cdot -\xi,\nabla \Big(\psi_{\mu,x}G_s(x,\cdot\,)\eta_k\Big)\rangle(-\Delta)^{s}\Big(\psi_{\gamma,y}G_s(y,\cdot\,)\eta_k\Big)\,dz.
\label{startingEquality}
\end{align}
Notice that there is no boundary term, since the functions $u$ and $v$ have compact support in $\Omega$. We pass to the limit as $k\rightarrow +\infty$, $\mu\rightarrow 0^+$, $\gamma\rightarrow 0^+$ on both sides of \eqref{startingEquality}. The proof is divided into 4 steps. We stress that, while \emph{Steps 1-2} hold for any $s\in (0,1)$ and any $\xi\in\mathbb R^N$, in \emph{Steps 3-4} we assume that $\xi=x$ when $s\leq \frac{1}{2}$. This is due to the integrability properties of $\nabla G_s(x,\cdot)$ (see Remark \ref{remark:IntegrabilityNablaG}).
\\\\
\emph{Step 1: We pass to the limit in the left hand side of \eqref{startingEquality}, showing that
\begin{equation}
\label{LHS_computation}
\lim_{\gamma\rightarrow 0^+}\lim_{\mu\rightarrow 0^+}\lim_{k\to+\infty}\Big(\mbox{LHS of \eqref{startingEquality}}\Big)=
(2s-N)G_s(y,x).\end{equation}}
By using the product rule \eqref{pl} we have 
\begin{equation}\label{psi-G-estim}\mbox{LHS of \eqref{startingEquality}}= (2s-N)A_{k,\mu,\gamma}(x,y) + (2s-N)B_{k,\mu,\gamma}(x,y)\end{equation}
where 
\begin{align*} 
A_{k,\mu,\gamma}(x,y)&:=\int_{\Omega}\psi_{\gamma,y}G_s(y,\cdot\,)\eta_k\Bigg(\psi_{\mu,x}G_s(x,\cdot\,)(-\Delta)^s\eta_k-\mathcal I_s[\eta_k,\psi_{\mu,x}G_s(x,\cdot\,)]\Bigg)dz;
\\
B_{k,\mu,\gamma}(x,y)&:=\int_{\Omega}\psi_{\gamma,y}G_s(y,\cdot\,)\eta_k^2(-\Delta)^{s}\Big(\psi_{\mu,x}G_s(x,\cdot\,)\Big)\,dz.\end{align*}
By the regularity results in \eqref{Regularity_PsiG}, we can apply Lemma \ref{lemma:similarFHFA}, with $w=\psi_{\gamma,y}G_s(y,\cdot\,)$, and $v=\psi_{\mu,x}G_s(x,\cdot\,)$, from which we deduce that, 
for any $\mu,\gamma>0$ fixed,  $\lim_{k\to+\infty}A_{k,\mu,\gamma}(x,y)=0.$
As a consequence one gets
\begin{align}\label{CovergenceFirstTermTOT}
\lim_{\gamma\rightarrow 0^+}\lim_{\mu\rightarrow 0^+}\lim_{k\to+\infty}A_{k,\mu,\gamma}(x,y)=0.
\end{align}
Furthermore, since $(-\Delta)^s\Big(\psi_{\mu,x}G_s(x,\cdot)\Big)\in L^{\infty}(\Omega)$ (see \eqref{fracLappsiGBounded}), by the dominated converge theorem we obtain
\begin{align}\label{firstLimitktobeused}
\lim_{k\to+\infty}B_{k,\mu,\gamma}(x,y)=\int_{\Omega}\psi_{\gamma,y}G_s(y,\cdot\,)(-\Delta)^{s}\Big(\psi_{\mu,x}G_s(x,\cdot\,)\Big)\,dz.
\end{align}
Using the product rule \eqref{pl}, and recalling that $\psi_{\mu,x}(-\Delta)^{s}G_s(x,\cdot\,)\equiv 0$, we have
\[
\int_{\Omega}\psi_{\gamma,y}G_s(y,\cdot\,)(-\Delta)^{s}\Big(\psi_{\mu,x}G_s(x,\cdot\,)\Big)\,dz
=\int_{\Omega}\psi_{\gamma,y}G_s(y,\cdot\,)\Bigg(G_s(x,\cdot\,)(-\Delta)^{s}\psi_{\mu,x}-\mathcal I_s[\psi_{\mu,x},G_s(x,\cdot\,)]\Bigg)\,dz.
\]
Next, we  pass to the limit as $\mu\rightarrow 0^+$, using Lemma \ref{lemma:lm2.2_DS_revisited} and Lemma \ref{psi_lm},  with $w=\psi_{\gamma,y}G_s(y,\cdot)$, getting
\begin{align*}
\lim_{\mu\to0^+}\int_{\Omega}\psi_{\gamma,y}G_s(y,\cdot\,)(-\Delta)^{s}\Big(\psi_{\mu,x}G_s(x,\cdot\,)\Big)\,dz&=\lim_{\mu\rightarrow 0^+} \int_{\Omega}\psi_{\gamma,y}G_s(y,\cdot\,)\Bigg(G_s(x,\cdot\,)(-\Delta)^{s}\psi_{\mu,x}-\mathcal I_s[\psi_{\mu,x},G_s(x,\cdot\,)]\Bigg)\,dz\nonumber\\
 &= \lim_{\mu\rightarrow 0^+} \int_{B_{\eps}(x)}\psi_{\gamma,y}G_s(y,\cdot\,)\Bigg(G_s(x,\cdot\,)(-\Delta)^{s}\psi_{\mu,x}-\mathcal I_s[\psi_{\mu,x},G_s(x,\cdot\,)]\Bigg)\,dz\nonumber\\
& +\lim_{\mu\rightarrow 0^+} \int_{\Omega\setminus B_\eps (x)}\psi_{\gamma,y}G_s(y,\cdot\,)\Bigg(G_s(x,\cdot\,)(-\Delta)^{s}\psi_{\mu,x}-\mathcal I_s[\psi_{\mu,x},G_s(x,\cdot\,)]\Bigg)\,dz\nonumber\\
 &= \psi_{\gamma,y}(x)G_s(y,x).
 \end{align*}
Passing to the limit as $\gamma\rightarrow 0^+$, we obtain
\[ \lim_{\gamma\to0^+}\lim_{\mu\to0^+}\int_{\Omega}\psi_{\gamma,y}G_s(y,\cdot\,)(-\Delta)^{s}\Big(\psi_{\mu,x}G_s(x,\cdot\,)\Big)\,dz=G_s(y,x),
\]
namely, by \eqref{firstLimitktobeused}
\begin{equation} \label{limitBgivedG}
\lim_{\gamma\to0^+}\lim_{\mu\to0^+}\lim_{k\to+\infty}B_{k,\mu,\gamma}(x,y) =G_s(y,x).
\end{equation}
In conclusion, \eqref{LHS_computation} follows from \eqref{limitBgivedG},  \eqref{CovergenceFirstTermTOT} and \eqref{psi-G-estim}, and this ends the proof of \emph{Step 1}.
\\
\\
\emph{Step 2: Let $s\in(0,1)$,  $\xi\in\mathbb R^N.$ We pass to the limit in the right hand side of \eqref{startingEquality}, showing that
\begin{align}
\label{RHS_computation}
\lim_{\gamma\rightarrow 0^+}\lim_{\mu\rightarrow 0^+}\lim_{k\to+\infty}\Big(\mbox{RHS of \eqref{startingEquality}}\Big)
=&
\, \langle x -\xi, \nabla_x G_s(y,x)\rangle 
 +\Gamma^2(1+s)  \int_{\partial\Omega} \frac{G_s(y,\cdot\,)}{\delta^s_{\Omega}}\frac{G_s(x,\cdot)}{\delta^s_{\Omega}}\, \langle\,\cdot -\xi, \nu\rangle \, d\sigma\nonumber\\
 +& \lim_{\gamma\rightarrow 0^+}\lim_{\mu\rightarrow 0^+} \int_{\Omega}
\langle\,\cdot -\xi, \nabla \Big( \psi_{\mu,x}G_s(x,\cdot)\Big)\rangle
(-\Delta)^s\Big(\psi_{\gamma,y}G_s(y,\cdot)\Big)dz.
\end{align}}
Using the product rule \eqref{pl} we have
\begin{equation}\label{RHSExpression}\mbox{RHS of \eqref{startingEquality}}= C_{k,\mu,\gamma}(x,y) + D_{k,\mu,\gamma}(x,y)+E_{k,\mu,\gamma}(x,y)\end{equation}
where 
\begin{align*} 
C_{k,\mu,\gamma}(x,y)&:=\int_{\Omega}
\langle \,\cdot -\xi,\nabla \Big(\psi_{\gamma,y}G_s(y,\cdot\,)\eta_k\Big)\rangle
\Bigg(\psi_{\mu,x}G_s(x,\cdot)(-\Delta)^s\eta_k-\mathcal I_s\big[\eta_k, G_s(x,\cdot)\psi_{\mu,x}\big]\Bigg)dz+
\\
&+\int_{\Omega}\langle \,\cdot -\xi,\nabla \Big(\psi_{\mu,x}G_s(x,\cdot\,)\eta_k\Big)\rangle
\Bigg(\psi_{\gamma,y}G_s(y,\cdot)(-\Delta)^s\eta_k-\mathcal I_s\big[\eta_k, G_s(y,\cdot)\psi_{\gamma,y}\big]\Bigg)dz;
\\
D_{k,\mu,\gamma}(x,y)&:=\int_{\Omega}\eta_k\langle\,\cdot -\xi, \nabla \Big(\eta_k \psi_{\gamma,y}G_s(y,\cdot)\Big)\rangle (-\Delta)^s\Big(\psi_{\mu,x}G_s(x,\cdot)\Big)dz;
\\
E_{k,\mu,\gamma}(x,y)&:=\int_{\Omega}\eta_k\langle\,\cdot -\xi, \nabla \Big(\eta_k \psi_{\mu,x}G_s(x,\cdot)\Big)\rangle (-\Delta)^s\Big(\psi_{\gamma,y}G_s(y,\cdot)\Big)dz.\end{align*}
We apply Lemma \ref{lemma_2.3_Djitte_Sueur_corrected} with  $\mathcal X=\cdot -\xi$, and $w=\psi_{\gamma,y}G_s(y,\cdot\,)$, $v=\psi_{\mu,x}G_s(x,\cdot)$ ($w,v$ satisfy all the regularity assumptions required, see \eqref{Regularity_PsiG}), 
to get
\[\lim_{k\rightarrow +\infty}C_{k,\mu,\gamma}(x,y)=\Gamma^2(1+s)  \int_{\partial\Omega} \frac{\psi_{\gamma,y}G_s(y,\cdot\,)}{\delta^s_{\Omega}}\frac{\psi_{\mu,x}G_s(x,\cdot)}{\delta^s_{\Omega}}\, \langle\,\cdot -\xi, \nu\rangle \, d\sigma.\]
Then we can pass to the limit in $\mu\rightarrow 0^+$ and $\gamma\rightarrow 0^+$, by the Lebesgue dominated convergence theorem, since both $x$ and $y$ are in the interior of $\Omega$, so
\begin{equation}\label{Atermlim}\lim_{\gamma\rightarrow 0^+}\lim_{\mu\rightarrow 0^+}\lim_{k\rightarrow +\infty}C_{k,\mu,\gamma}(x,y)=\Gamma^2(1+s)  \int_{\partial\Omega} \frac{G_s(y,\cdot\,)}{\delta^s_{\Omega}}\frac{G_s(x,\cdot)}{\delta^s_{\Omega}}\, \langle\,\cdot -\xi, \nu\rangle \, d\sigma.\end{equation}
Furthermore, we apply Lemma \ref{lemma:lm_2.4_DS_revisited} with $\mathcal X=\cdot -\xi$, and $v=\psi_{\gamma,y}G_s(y,\cdot)$ ($v$ satisfies the required assumptions, see \eqref{Regularity_PsiG}-\eqref{eq:nablaPsiGL1}), to get
\[\lim_{\mu\rightarrow 0^+}\lim_{k\rightarrow +\infty}D_{k,\mu,\gamma}(x,y)= \langle x -\xi, \nabla_x \Big(\psi_{\gamma,y}(x)G_s(y,x)\Big)\rangle= G_s(y,x)\langle x -\xi, \nabla \psi_{\gamma,y}(x)\rangle+ \psi_{\gamma,y}(x) \langle x -\xi, \nabla_x G_s(y,x)\rangle.\]
Hence, passing to the limit in $\gamma\rightarrow 0^+$, since for $\gamma>0$ sufficiently small $\psi_{\gamma,y}(x) = 1$ and $\nabla \psi_{\gamma,y}(x) = 0$, we have
\begin{equation}\label{Btermlim}\lim_{\gamma\rightarrow 0^+}\lim_{\mu\rightarrow 0^+}\lim_{k\rightarrow +\infty}D_{k,\mu,\gamma}(x,y)= \langle x -\xi, \nabla_x G_s(y,x)\rangle.\end{equation}
Finally, we consider the term $E_{k,\mu,\gamma}(x,y)$. By Leibniz rule, we split it as 
 \begin{equation}\label{dedede}
 E_{k,\mu,\gamma}(x,y) =E^1_{k,\mu,\gamma}(x,y)+E^2_{k,\mu,\gamma}(x,y) , 
 \end{equation}
 where 
\begin{align*} 
E^1_{k,\mu,\gamma}(x,y) &:= \int_{\Omega}\eta_k \psi_{\mu,x}G_s(x,\cdot)\langle \cdot-\xi,\nabla\eta_k\rangle (-\Delta)^s\Big(\psi_{\gamma,y}G_s(y,\cdot)\Big)dz, \nonumber\\
E^2_{k,\mu,\gamma}(x,y) &:= 
\int_{\Omega}\eta_k^2
\langle\,\cdot -\xi, \nabla \Big( \psi_{\mu,x}G_s(x,\cdot)\Big)\rangle
(-\Delta)^s\Big(\psi_{\gamma,y}G_s(y,\cdot)\Big)dz .
\end{align*}
Using Lemma \ref{lemma:lm_2.4_DS_revisitedFIRST_PART}  with  $v=\psi_{\mu,x}G_s(x,\cdot)$, $w=(-\Delta)^s\Big(\psi_{\gamma,y}G_s(y,\cdot)\Big)$ (the assumptions on $v$ and $w$ are satisfied thanks to  \eqref{fracLappsiGBounded}-\eqref{Regularity_PsiG}), and $\mathcal X=\cdot -\xi$, we  deduce that
$
\lim_{k\to+\infty} E^1_{k,\mu,\gamma}(x,y)=0
$, for any $\mu,\gamma>0$ small,
hence
\begin{equation}\label{limC1}
  \lim_{\gamma\to 0^+}\lim_{\mu\to 0^+}\lim_{k\to+\infty}E^1_{k,\mu,\gamma}(x,y)=0.
\end{equation}
For the term $E^2_{k,\mu,\gamma}(x,y)$, 
we first pass to the limit as $k\rightarrow +\infty$, using the dominated convergence theorem. Indeed  $ \nabla \Big( \psi_{\mu,x}G_s(x,\cdot)\Big)\in L^1(\Omega)$ (see \eqref{eq:nablaPsiGL1}), and  $|\cdot-\xi|(-\Delta)^s\Big(\psi_{\gamma,y}G_s(y,\cdot)\Big)\in L^{\infty}(\Omega)$ (see \eqref{fracLappsiGBounded}), hence 
\begin{equation}\label{intermediateLimitmk0}
\lim_{k\rightarrow +\infty}E^2_{k,\mu,\gamma}(x,y)=\int_{\Omega}
\langle\,\cdot -\xi, \nabla \Big( \psi_{\mu,x}G_s(x,\cdot)\Big)\rangle
(-\Delta)^s\Big(\psi_{\gamma,y}G_s(y,\cdot)\Big)dz. \end{equation}
From \eqref{intermediateLimitmk0}-\eqref{limC1}-\eqref{dedede}-\eqref{Btermlim}-\eqref{Atermlim}-\eqref{RHSExpression} we deduce \eqref{RHS_computation}, and this concludes the proof of \emph{Step 2.}\\
\\
\emph{Step 3. Let $s\in (0,1)$. Let us fix $\xi:=x$ when $s\leq \frac{1}{2}$. We show that
\begin{equation}\label{Ctermlim}\lim_{\gamma\rightarrow 0^+}\lim_{\mu\rightarrow 0^+} \int_{\Omega}
\langle\,\cdot -\xi, \nabla \Big( \psi_{\mu,x}G_s(x,\cdot)\Big)\rangle
(-\Delta)^s\Big(\psi_{\gamma,y}G_s(y,\cdot)\Big)dz= \langle y -\xi, \nabla_y G_s(x,y)\rangle.\end{equation}
}
By the Leibniz rule we get
\begin{align}
\label{intermediateLimitmuA}
\nonumber\int_{\Omega}
\langle\,\cdot -\xi, \nabla \Big( \psi_{\mu,x}G_s(x,\cdot)\Big)\rangle
(-\Delta)^s\Big(\psi_{\gamma,y}G_s(y,\cdot)\Big)dz&=
\int_{\Omega} \psi_{\mu,x}
\langle\,\cdot -\xi, \nabla G_s(x,\cdot)\rangle
(-\Delta)^s\Big(\psi_{\gamma,y}G_s(y,\cdot)\Big)dz
\\
&+\int_{\Omega} G_s(x,\cdot)
\langle\,\cdot -\xi, \nabla \psi_{\mu,x}\rangle
(-\Delta)^s\Big(\psi_{\gamma,y}G_s(y,\cdot)\Big)dz.
\end{align}
We first pass to the limit as $\mu\rightarrow 0^+$ in the two terms separately.
For the first term we can pass to the limit as $\mu\rightarrow 0^+$ by the dominated convergence theorem. Indeed,  $(-\Delta)^s\Big(\psi_{\gamma,y}G_s(y,\cdot)\Big)\in L^{\infty}(\Omega)$ (see \eqref{fracLappsiGBounded})
and, under our assumption on $s$ and $\xi$, the function $\langle  \cdot-\xi , \nabla G_s(x,\cdot)\rangle\in L^1(\Omega)$,  
as already observed in Remark \ref{remark:IntegrabilityNablaG} (see \eqref{eq:GradGL1} and \eqref{eq:FieldGradGL^1}). 
Hence
\begin{equation}\label{intermediateLimitmuB}\lim_{\mu\rightarrow 0^+}\int_{\Omega} \psi_{\mu,x}
\langle\,\cdot -\xi, \nabla G_s(x,\cdot)\rangle
(-\Delta)^s\Big(\psi_{\gamma,y}G_s(y,\cdot)\Big)dz=\int_{\Omega} 
\langle\,\cdot -\xi, \nabla G_s(x,\cdot)\rangle
(-\Delta)^s\Big(\psi_{\gamma,y}G_s(y,\cdot)\Big)dz\end{equation}
For the second term we compute $\nabla\psi_{\mu,x}$, and use again that $(-\Delta)^s\Big(\psi_{\gamma,y}G_s(y,\cdot)\Big)\in L^{\infty}(\Omega)$, getting
\begin{align}\nonumber
\left|\int_{\Omega} G_s(x,\cdot)
\langle\,\cdot -\xi, \nabla \psi_{\mu,x}\rangle
(-\Delta)^s\Big(\psi_{\gamma,y}G_s(y,\cdot)\Big)dz\right|
&\leq C(\gamma,y)  
\int_{A_\mu(x)} G_s(x,\cdot)
\frac{8}{\delta^2_\Omega(x)}\frac{|\xi-\cdot|\, |x-\cdot|}{\mu^2}
\left| 
\rho'\Big(\frac{8}{\delta^2_\Omega(x)}\frac{|x-\cdot|^2}{\mu^2}\Big)\right|\,dz\\\nonumber
&\leq C(\gamma,y)  
\int_{A_\mu(x)} G_s(x,\cdot)\frac{|\xi-\cdot|}{|x-\cdot|}
\,dz,
\\\label{intermediateLimitmuC}
&\leq  C(\gamma,y)  \Lambda_{N,s}
\int_{A_\mu(x)} \frac{|\xi-\cdot|}{|x-\cdot|^{N-2s+1}}
\,dz\longrightarrow 0\quad \mbox{ as }\mu\rightarrow 0^+,
\end{align}
where we have used that $\frac{8}{\delta^2_\Omega(x)}\frac{|x-z|^2}{\mu^2}\leq 2$ for $z\in A_\mu(x)$, with $A_\mu(x):=\{z\in\Omega\, :\, \mu\frac{\delta(x)}{2\sqrt 2}\leq|z-x|\leq \mu\frac{\delta(x)}{2}\}$ and, in the last inequality, also the estimate on $G_s(x,\cdot)$ in \eqref{greenfunctEstimate}. We stress that the convergence in \eqref{intermediateLimitmuC}  holds under our assumption on $s$ and $\xi$.
Substituting \eqref{intermediateLimitmuC} and \eqref{intermediateLimitmuB} into \eqref{intermediateLimitmuA} we get 
\begin{equation}
    \label{intermediateLimitmu}
  \lim_{\mu\rightarrow 0^+}\int_{\Omega}
\langle\,\cdot -\xi, \nabla \Big( \psi_{\mu,x}G_s(x,\cdot)\Big)\rangle
(-\Delta)^s\Big(\psi_{\gamma,y}G_s(y,\cdot)\Big)dz=\int_{\Omega}
\langle\,\cdot -\xi, \nabla G_s(x,\cdot)\rangle
(-\Delta)^s\Big(\psi_{\gamma,y}G_s(y,\cdot)\Big)dz. 
\end{equation}
Next, we apply the fractional product law \eqref{pl} to the right hand side of \eqref{intermediateLimitmu}, using that  $\psi_{\gamma,y}(-\Delta)^s G_s(y,\cdot)=0$, and then pass to the limit as $\gamma\rightarrow 0^+$, and get \eqref{Ctermlim}:
\begin{align}\label{quasiQuasiLi}
\lim_{\gamma\rightarrow 0^+}\int_{\Omega}
\langle\,\cdot -\xi, \nabla G_s(x,\cdot)\rangle
(-\Delta)^s\Big(\psi_{\gamma,y}G_s(y,\cdot)\Big)dz&= \lim_{\gamma\rightarrow 0^+}
\int_{\Omega}
\langle\,\cdot -\xi, \nabla G_s(x,\cdot)\rangle
 \Bigg(G_s(y,\cdot)(-\Delta)^s\psi_{\gamma,y}-\mathcal I_s\big[G_s(y,\cdot),\psi_{\gamma,y}\big]\Bigg)dz\nonumber\\
 &= \lim_{\gamma\rightarrow 0^+}
\int_{B_{\varepsilon}(y)}
\langle\,\cdot -\xi, \nabla G_s(x,\cdot)\rangle
 \Bigg(G_s(y,\cdot)(-\Delta)^s\psi_{\gamma,y}-\mathcal I_s\big[G_s(y,\cdot),\psi_{\gamma,y}\big]\Bigg)dz\nonumber\\
& +\lim_{\gamma\rightarrow 0^+}
\int_{\Omega\setminus B_{\varepsilon}(y)}
\langle\,\cdot -\xi, \nabla G_s(x,\cdot)\rangle
 \Bigg(G_s(y,\cdot)(-\Delta)^s\psi_{\gamma,y}-\mathcal I_s\big[G_s(y,\cdot),\psi_{\gamma,y}\big]\Bigg)dz\nonumber\\
 &= \langle  y-\xi , \nabla_y G_s(x,y)\rangle.
 \end{align}
Notice that in \eqref{quasiQuasiLi} we passed to the limit using Lemma \ref{lemma:lm2.2_DS_revisited} and Lemma \ref{psi_lm}  with $w=\langle  \cdot-\xi , \nabla G_s(x,\cdot)\rangle$. Obviously $w\in  C(\overline{B_{\varepsilon}(y)})$. Furthermore, our restriction either on $s$ or on $\xi$, ensures that $w\in L^1(\Omega\setminus B_{\varepsilon}(y))$. Hence $w$ satisfies the assumptions of Lemma \ref{lemma:lm2.2_DS_revisited} and Lemma \ref{psi_lm}.\\
\\
\emph{Step 4. Conclusion of the proof.}\\
Let $s\in (0,1)$. Let us fix $\xi:=x$ when $s\leq \frac{1}{2}$. By \emph{Step  2} and \emph{Step  3}, we can pass to the limits in the right hand side of \eqref{startingEquality}, and conclude that
\begin{align*}
\lim_{\gamma\rightarrow 0^+}\lim_{\mu\rightarrow 0^+}\lim_{k\to+\infty}\Big(\mbox{RHS of \eqref{startingEquality}}\Big)
&= \Gamma^2(1+s)  \int_{\partial\Omega} \frac{G_s(y,\cdot\,)}{\delta^s_{\Omega}}\frac{G_s(x,\cdot)}{\delta^s_{\Omega}}\, \langle\,\cdot -\xi, \nu\rangle \, d\sigma\\
&+ \langle x -\xi, \nabla_x G_s(y,x)\rangle
+  \langle y -\xi, \nabla_y G_s(x,y)\rangle
.
\end{align*}

As a consequence, by \emph{Step 1}, we have
\begin{align}\nonumber
-(N-2s)G_s(y,x)
&= \Gamma^2(1+s)  \int_{\partial\Omega} \frac{G_s(y,\cdot\,)}{\delta^s_{\Omega}}\frac{G_s(x,\cdot)}{\delta^s_{\Omega}}\, \langle\,\cdot -\xi, \nu\rangle \, d\sigma +\langle x -\xi, \nabla_x G_s(y,x)\rangle
+  \langle y -\xi, \nabla_y G_s(x,y)\rangle
.\\\nonumber
&= \Gamma^2(1+s)  \int_{\partial\Omega} \frac{G_s(y,\cdot\,)}{\delta^s_{\Omega}}\frac{G_s(x,\cdot)}{\delta^s_{\Omega}}\, \langle\,\cdot -\xi, \nu\rangle \, d\sigma -\langle x -\xi, \nabla_x H_s(y,x)\rangle
-  \langle y -\xi, \nabla_y H_s(x,y)\rangle\\\label{OneBeforeTheLast}
&+ \langle x -\xi, \nabla_x F_s(y,x)\rangle
+  \langle y -\xi, \nabla_y F_s(x,y)\rangle,
\end{align}
where we have split  the fractional Green function as in \eqref{Eq-spliting-of-Green}.  We claim that 
\begin{equation}\label{finalCLAIM}
\langle x -\xi, \nabla_x F_s(y,x)\rangle
+  \langle y -\xi, \nabla_y F_s(x,y)\rangle= -(N-2s)F_s(x,y).
\end{equation}
By \eqref{funda} and a direct computation, we have
\[
\nabla_x F_s(y,x)=\frac{-(N-2s)b_{N,s}(x-y)}{|x-y|^{N-2s+2}}=-\nabla_y F_s(x,y),
\]
hence \[\langle x-y, \nabla_x F_s(y,x)\rangle= (2s-N)F_s(x,y),\]
and the claim \eqref{finalCLAIM} easily follows:
\begin{align*}
\langle x -\xi, \nabla_x F_s(y,x)\rangle
+  \langle y -\xi, \nabla_y F_s(x,y)\rangle&= \langle x-y, \nabla_x F_s(y,x)\rangle+\langle y -\xi, \nabla_x F_s(y,x)\rangle + \langle y -\xi, \nabla_y F_s(x,y)\rangle\\
&= \langle x-y, \nabla_x F_s(y,x)\rangle\\
&=(2s-N)F_s(x,y).
\end{align*}
Substituting  \eqref{finalCLAIM} into \eqref{OneBeforeTheLast} we deduce 
\begin{align}\nonumber
(N-2s)\big(F_s(x,y)-G_s(y,x)\big)
&= \Gamma^2(1+s)  \int_{\partial\Omega} \frac{G_s(y,\cdot\,)}{\delta^s_{\Omega}}\frac{G_s(x,\cdot)}{\delta^s_{\Omega}}\, \langle\,\cdot -\xi, \nu\rangle \, d\sigma -\langle x -\xi, \nabla_x H_s(y,x)\rangle
-  \langle y -\xi, \nabla_y H_s(x,y)\rangle.
\end{align}
Recalling the splitting \eqref{Eq-spliting-of-Green},  and that we are considering any $\xi\in\mathbb R^N$ when $s>\frac{1}{2}$, and $\xi=x$ when $s\leq \frac{1}{2}$, we derive both
\eqref{bilinearPohozaevGreenGENERAL}  when $s>\frac{1}{2}$, and \eqref{bilinearPohozaevGreen} for $s\leq \frac{1}{2}$. Finally taking $\xi=x$ in \eqref{bilinearPohozaevGreenGENERAL}, we obtain \eqref{bilinearPohozaevGreen} in the full range $s\in (0,1)$.\\
\\
\begin{remark}
When $s>\frac{1}{2}$,
taking \eqref{bilinearPohozaevGreenGENERAL} with $\xi=y$, and $\xi=x$, respectively, and subtracting term by term, we deduce in particular that
\[\Gamma^2(1+s)\int_{\partial\Omega}\frac{ G_s(x,\cdot)}{\delta^s}\frac{G_s(y,\cdot)}{\delta^s} \langle x-y,\nu\rangle \, d\sigma=
   \langle \nabla_x  H_s(y,x)+ \nabla_y  H_s(x,y) ,x-y\rangle.
\]
$\,$\\
\end{remark}
We end the section with the 
\begin{proof}[Proof of Theorem \ref{MainTheorem}]
It is enough to observe that, fixing $x\in\Omega$ in \eqref{bilinearPohozaevGreen}, and  letting $y$ go to $x$, \eqref{PohozaevGreen} follows. 
\end{proof}

\section{The local case $s=1$}
Throughout the section, for $x\in\Omega$, we denote by $G_1(x,\cdot)$ the Green function associated to classical Laplacian operator $-\Delta$ in the domain $\Omega$, namely the solution to \[
\left\{ \begin{array}{rcll}
     -\Delta G_1(x,\cdot)&=&\delta_x& \quad\text{in}\quad \mathcal D'(\Omega)\\
    G_1(x,\cdot)&=&0&\quad\text{on}\quad\partial\Omega,
    \end{array} 
    \right.
\]
where $\delta_x$ denotes the Dirac delta distribution at $x$. The Green function $G_1(x,\cdot)$ can be split into
\[
    G_1(x,\cdot)= F_1(x,\cdot)-H_1(x,\cdot),
\]
where 
\[
   F_1(x,\cdot) :=\frac{1}{(N-2)\sigma_N|\cdot - x|^{N-2}}, 
\]
 is the  fundamental solution of $-\Delta$. Here  $\sigma_N:=|\mathbb{S}^{N-1}|$.
As a consequence, 
$H_1(x,\cdot)$ solves the equation
\[
\left\{ \begin{array}{rcll} -\Delta H_1(x,\cdot)&=& 0  &\textrm{in }\Omega, \\ H_1(x,\cdot)&=&F_1(x,\cdot)& 
\textrm{on }\partial\Omega. \end{array}\right. 
\]
We recall the following classical generalized Pohozaev identity, for functions  $u,v\in C^2(\overline\Omega)$, $u=v=0$ on $\partial\Omega$
\begin{align}\label{eq:pohozaevLocal}
\int_\Omega\langle \cdot-\xi,\nabla u\rangle(-\Delta v)\,dx+\int_\Omega \langle \cdot-\xi,\nabla v\rangle(-\Delta u)\,dx
=(2-N)\int_{\Omega}u(-\Delta v)dx-\int_{\partial\Omega}\frac{ \partial u}{\partial\nu}\frac{\partial v}{\partial\nu}\langle \cdot-\xi,\nu\rangle\,d\sigma,
\end{align}
where $\nu$ is the outward unit normal to $\partial\Omega$, and $\xi\in\mathbb R^N$, $N\geq 2$.
\\\\
Taking inspiration from \cite{BrezisPeletier}, for $a\in\R^N$, we define the family of functions 
\begin{equation}\label{def:deltarho}
\delta_{\rho,a}:=\frac{1}{|B_{\rho}(a)|} \chi_{B_\rho(a)}=\frac{N}{\sigma_N}\rho^{-N} \chi_{B_\rho(a)}, \quad \rho>0 
\end{equation}
where $B_\rho(a):=\{z\in\mathbb R^N\,:\, |z-a|<\rho\}$, and $\chi_{B_\rho(a)}$ denotes the characteristic function of the ball $B_\rho(a)$. 
The functions $\delta_{\rho,a}$ converge  to $\delta_a$, the Dirac delta distribution at $a$, in the sense of distributions.\\
\\
Let also $v_{\rho,a}$ be the solution to the equation
\begin{align}\label{eqvrho}
-\Delta v_{\rho,a}=\delta_{\rho,a} \;\text{ in } \mathbb{R}^N,
\end{align}
 such that $v_{\rho,a}\to0$ as $x\to+\infty$. Then we have
\begin{equation}\label{def:vrho}
v_{\rho,a} (z):=\left\{ \begin{array}{ll}
\displaystyle{-\frac{|z-a|^2}{2\sigma_N\rho^N}+\frac{N}{2(N-2)\sigma_N\rho^{N-2}}}& \text{ if } 0<|z-a|<\rho \\
&\\
F_1(a,z) & \text{ if } \rho\leq|z-a|<+\infty.
\end{array}\right. 
\end{equation}
We also let $u_{\rho,a}$ be the solution of the problem
\begin{equation}\label{u_def}
\left\{ \begin{array}{rcll} -\Delta u_{\rho,a}&=& \delta_{\rho,a}  &\textrm{in }\Omega \\ u_{\rho,a}&=& 0&\textrm{on }\partial\Omega. \end{array}\right. 
\end{equation}
Let us observe that for any neighborhood $\omega$ of $\partial\Omega$ which does not contain the point $a$
\begin{equation}\label{uApproxG}u_{\rho,a}\rightarrow G_1(a,\cdot)\mbox{ in }C^k(\omega), \ \forall k\in\mathbb N.\end{equation}

\begin{proof}[Proof of Theorem \ref{Main_local_TheoremBilinear_s=1}]
 We apply the Pohozaev identity \eqref{eq:pohozaevLocal} to the regular functions $u= u_{\rho,x}$ and $v=u_{\rho,y}$, with $x\neq y$, for any fixed $\xi\in\mathbb R^N$, and get 
\begin{equation}\label{Starting_s=1_Equal}
\int_{\partial\Omega}\frac{ \partial u_{\rho,x}}{\partial\nu}\frac{\partial u_{\rho,y}}{\partial\nu}\langle \cdot-\xi,\nu\rangle\,d\sigma=-\int_\Omega \delta_{\rho,y}\langle \cdot-\xi ,\nabla u_{\rho,x}\rangle\,dz-\int_\Omega \delta_{\rho,x}\langle \cdot-\xi,\nabla u_{\rho,y}\rangle\,dz+(2-N)\int_{\Omega}u_{\rho,x} \delta_{\rho,y}dz.
\end{equation}
We pass to the limit as $\rho\rightarrow 0^+$.
For the left hand side, since both $x$ and $y$ do not belong to $\partial\Omega$, by \eqref{uApproxG} we have
\begin{equation}\label{RHSLocal}
\lim_{\rho\to0^+}\int_{\partial\Omega}\frac{ \partial u_{\rho,x}}{\partial\nu}\frac{\partial u_{\rho,y}}{\partial\nu}\langle \cdot-\xi,\nu\rangle\,d\sigma=\int_{\partial\Omega}\frac{ \partial G_1(x,\cdot)}{\partial\nu}\frac{\partial G_1(y,\cdot)}{\partial\nu} \langle \cdot-\xi,\nu\rangle \, d\sigma.
\end{equation}
To evaluate the right hand side of \eqref{Starting_s=1_Equal}, taking inspiration from \cite[Proof of Theorem 4.3]{BrezisPeletier}, we rewrite it as
\begin{equation}\label{rhsDivision}
\mbox{RHS of \eqref{Starting_s=1_Equal}} =  \widetilde A_{\rho} (x,y) + \widetilde  B_\rho(x,y),
\end{equation}
where
\begin{align*}
&\widetilde A_{\rho} (x,y):=-\int_\Omega \delta_{\rho,y}\langle \cdot-\xi,\nabla (u_{\rho,x}-v_{\rho,x})\rangle\,dz-\int_\Omega \delta_{\rho,x}\langle\cdot-\xi,\nabla (u_{\rho,y}
-v_{\rho,y})\rangle\,dz+(2-N)\int_{\Omega}(u_{\rho,x}-v_{\rho,x}) \delta_{\rho,y}\,dz,\\
&\widetilde B_\rho(x,y):=-\int_\Omega \delta_{\rho,y}\langle \cdot-\xi,\nabla v_{\rho,x} \rangle\,dz-\int_\Omega \delta_{\rho,x}\langle\cdot-\xi,\nabla v_{\rho,y}\rangle\,dz+(2-N)\int_{\Omega}v_{\rho,x}\delta_{\rho,y}\,dz.
\end{align*}
First we show that
\begin{equation}\label{Step1Local}\lim_{\rho\rightarrow 0^+}\widetilde B_\rho(x,y)=0.\end{equation}
Let us observe that, since $x\neq y$, by \eqref{def:vrho} it follows that for $\rho>0$ sufficiently small
 \begin{equation}\label{vrho=F}v_{\rho,x}(z)=\frac{1}{(N-2)\sigma_N|z-x|^{N-2}},\qquad \mbox{for }z\in B_\rho(y),\end{equation}
 hence  by direct computations
\begin{equation}\label{ExplComputNablavrho}\nabla v_{\rho,x}(z)=-\frac{1}{\sigma_N}\frac{z-x}{|z-x|^N},\qquad \mbox{for }z\in B_\rho(y).\end{equation}
Using the definition \eqref{def:deltarho} of $\delta_{\rho,y}$, and  \eqref{ExplComputNablavrho}, one has
\begin{align}\label{grad_v_x}
-\int_\Omega \delta_{\rho,y}\langle z-\xi, \nabla v_{\rho,x}\rangle\,dz&=-\frac{N}{\sigma_N}\rho^{-N}\int_{B_{\rho}(y)}\langle z-\xi, \nabla v_{\rho,x}\rangle\,dz
=\frac{N}{\sigma_N^2}\rho^{-N}\int_{B_{\rho}(y)}\frac{\langle z -\xi,  z -x\rangle}{|z -x|^N}\,dz\nonumber\\
& =\frac{N}{\sigma_N^2}\int_{B_{1}(0)}\frac{\langle  \rho t+y-\xi, \rho t+y-x\rangle}{|\rho t+y-x|^N}\,dt\ \longrightarrow \ \frac{1}{\sigma_N}\frac{\langle  y-\xi,  y-x \rangle}{|x-y|^N},
\end{align}
as $\rho\rightarrow 0^+$, where we have changed coordinates $z=\rho t+y$, and then passed to the limit by the dominated convergence theorem. Arguing similarly, one also gets
\begin{equation}\label{grad_v_y}
-\int_\Omega \delta_{\rho,x}\langle z-\xi,\nabla v_{\rho,y}\rangle\,dz\ \longrightarrow\ \frac{1}{\sigma_N}\frac{\langle x-\xi, x-y\rangle}{|x-y|^N}.
\end{equation}
as $\rho\rightarrow 0^+$. Furthermore, by the definition \eqref{def:deltarho} of $\delta_{\rho,y}$ and the observation \eqref{vrho=F}, it also follows
\begin{align}\label{lastTermConver}
(2-N)\int_{\Omega}v_{\rho,x} \delta_{\rho,y}\,dz=-\frac{N}{\sigma_N^2}\rho^{-N}\int_{B_{\rho}(y)}\frac{1}{|z-x|^{N-2}}\,dz
=-\frac{N}{\sigma_N^2}\int_{B_{1}(0)}\frac{1}{|\rho t+y-x|^{N-2}}\,dt\ \longrightarrow -\frac{1}{\sigma_N}\frac{1}{|y-x|^{N-2}},
\end{align}
as $\rho\rightarrow 0^+$, where we could pass to the limit, using the Lebesgue convergence theorem.
By \eqref{lastTermConver}, \eqref{grad_v_y} and \eqref{grad_v_x} we conclude that
\[\lim_{\rho\rightarrow 0^+}\widetilde B_\rho(x,y)=\frac{1}{\sigma_N}\frac{\langle  y-\xi,  y-x \rangle}{|x-y|^N}+ \frac{1}{\sigma_N}\frac{\langle x-\xi, x-y\rangle}{|x-y|^N}-\frac{\langle x-y, x-y\rangle}{\sigma_N}\frac{1}{|y-x|^{N}}=0,\]
where the last equality follows by trivial computation. This ends the proof of \eqref{Step1Local}.\\
\\
Next we show that 
\begin{equation}
\label{limitAtilde}
\lim_{\rho\rightarrow 0^+}\widetilde A_{\rho} (x,y)=(N-2)H_1(x,y)+\langle y-\xi,\nabla_y H_1(x,y)\rangle+\langle x-\xi,\nabla_x H_1(y,x)\rangle
\end{equation}

Let us observe that by \eqref{u_def} and \eqref{eqvrho},  the function  $u_{\rho,x}-v_{\rho,x}$ solves
\[
\left\{ \begin{array}{rcll} -\Delta (u_{\rho,x}-v_{\rho,x})&=& 0  &\textrm{in }\Omega \\ u_{\rho,x}-v_{\rho,x}&=& -v_{\rho,x}&\textrm{on }\partial\Omega. \end{array}\right. 
\]
Furthermore, since $x\not\in\partial\Omega$, by \eqref{def:vrho} it follows that for $\rho>0$ sufficiently small
 \[v_{\rho,x}(z)=F_1(x,z),\qquad \mbox{for }z\in\partial\Omega.\]
As a consequence, for $\rho>0$ sufficiently small
 $u_{\rho,x}-v_{\rho,x}$ solves
\[
\left\{ \begin{array}{rcll} -\Delta (u_{\rho,x}-v_{\rho,x})&=& 0  &\textrm{in }\Omega \\ u_{\rho,x}-v_{\rho,x}&=& -F_1(x,\cdot)&\textrm{on }\partial\Omega, \end{array}\right. 
\]
namely, 
\begin{equation}\label{FondamentalObservation}
u_{\rho,x}-v_{\rho,x}=-H_1(x,\cdot), \quad \mbox{ for }\rho>0 \mbox{ small enough}.
\end{equation}
Using the definition \eqref{def:deltarho} of $\delta_{\rho,y}$, and \eqref{FondamentalObservation}, we deduce
\begin{align}\label{firstTermFinalLimit}
    (2-N)\int_{\Omega}(u_{\rho,x}-v_{\rho,x}) \delta_{\rho,y}\,dz &= (N-2) \frac{N}{\sigma_N}\rho^{-N}\int_{B_\rho (y)} H_1(x,z)\, dz\nonumber\\
    &=(N-2) \frac{N}{\sigma_N}\int_{B_1(0)} H_1(x,\rho t+y)\, dt \ \longrightarrow (N-2)H_1(x,y),
\end{align}
and
\begin{align}\label{secondTermFinalLimit}
-\int_\Omega \delta_{\rho,y}\langle \cdot-\xi,\nabla (u_{\rho,x}-v_{\rho,x})\rangle\,dz&=
\frac{N}{\sigma_N}\rho^{-N}\int_{B_\rho (y)} \langle z-\xi,\nabla_z H_1(x,z)\rangle\, dz
\nonumber\\
    &=\frac{N}{\sigma_N}\int_{B_1(0)} \langle \rho t+y-\xi,\nabla_z H_1(x,\rho t+y)\rangle\, dt\ \longrightarrow \ \langle y-\xi,\nabla_y H_1(x,y)\rangle,
\end{align}
as $\rho\rightarrow 0^+$. 
Similarly to \eqref{secondTermFinalLimit}, we also get
\begin{equation}\label{lastTermFinalLimit}-\int_\Omega \delta_{\rho,x}\langle\cdot-\xi,\nabla (u_{\rho,y}
-v_{\rho,y})\rangle\,dz\ \longrightarrow\ \langle x-\xi,\nabla_x H_1(y,x)\rangle.\end{equation}
By \eqref{lastTermFinalLimit}, \eqref{secondTermFinalLimit} and \eqref{firstTermFinalLimit}, we get \eqref{limitAtilde}.\\\\
Finally, \eqref{limitAtilde}, \eqref{Step1Local}, 
\eqref{rhsDivision} and  \eqref{RHSLocal}, conclude the proof of Theorem \ref{Main_local_TheoremBilinear_s=1}.
\end{proof}
\begin{remark}
    With similar argument one could also prove for  $N\in\mathbb N$, $N > 2$, and $\Omega$ a bounded open set of $\mathbb R^N$ of class $C^{1,1}$, that for $x,y\in\Omega$ with $x\neq y$, it holds
\[
\displaystyle\int_{\partial\Omega}\frac{ \partial G_1(x,\cdot)}{\partial\nu}\frac{\partial G_1(y,\cdot)}{\partial\nu} \nu_i \, d\sigma=  \frac{\partial}{\partial x_i}H_1(y,x) 
+   \frac{\partial}{\partial y_i} H_1(x,y).
\]
This is the counterpart in the case $s=1$ of \cite[Theorem 1.4]{DjitteSueur}.  Furthermore, fixing $x$ and letting $y$ go to $x$, it gives the Pohozaev identity for the classical Green function already obtained in \cite[Theorem 4.4]{BrezisPeletier}.
\end{remark}

\bibliographystyle{NAplain}

\begin{flushleft}
\vspace{1cm}
\textbf{Abdelrazek Dieb}\\
Department of Mathematics, Faculty of Mathematics and Computer Science \\
University Ibn Khaldoun of Tiaret\\
Tiaret 14000, Algeria\\
and\\
Laboratoire d’Analyse Nonlinéaire et Mathématiques Appliquées\\
Université Abou Bakr Belkaïd,Tlemcen, Algeria\\
\texttt{abdelrazek.dieb@univ-tiaret.dz}
\vspace{.5cm}

\textbf{Isabella Ianni}\\
Dipartimento di Scienze di Base e Applicate per l’Ingegneria\\
{\sl Sapienza} Universit\`a di Roma\\
Via Scarpa 16, 00161 Roma, Italy\\
\texttt{isabella.ianni@uniroma1.it} 
\vspace{.3cm}
\end{flushleft}

\end{document}